\documentclass[11pt, reqno]{amsart}
\usepackage{indentfirst, amssymb, amsmath, amsthm, mathrsfs, setspace, indentfirst, enumerate,  mathrsfs, amsmath, amsthm}
\usepackage[bookmarksnumbered, colorlinks, plainpages]{hyperref}
\usepackage{mathrsfs}
\usepackage{cite}
\usepackage{graphicx}
\usepackage{float}
\newtheorem*{theo4.A}{Theorem 4.A}
\newtheorem*{theo4.B}{Theorem 4.B}

\newtheorem*{theoA}{Theorem A}
\newtheorem*{theoB}{Theorem B}
\newtheorem*{theoC}{Theorem C}
\newtheorem*{theoD}{Theorem D}
\newtheorem*{theoE}{Theorem E}

\newtheorem*{theo2.E}{Theorem 2.E}
\newtheorem*{theo2.F}{Theorem 2.F}

\newtheorem*{theo3.A}{Theorem 3.A}
\newtheorem*{theo3.B}{Theorem 3.B}
\newtheorem*{ques3.A}{Question 3.A}

\newtheorem*{cor A}{Corollary A}
\newtheorem*{cor B}{Corollary B}
\newtheorem*{que1}{Question 1.1}
\newtheorem*{que2}{Question 1.2}

\newtheorem{theo}{Theorem}[section]
\newtheorem{lem}{Lemma}[section]

\newtheorem{defi}{Definition}[section]
\newtheorem{rem}{Remark}[section]

\newcommand{\ol}{\overline}
\newcommand{\be}{\begin{equation}}
\newcommand{\ee}{\end{equation}}
\newcommand{\beas}{\begin{eqnarray*}}
\newcommand{\eeas}{\end{eqnarray*}}
\newcommand{\bea}{\begin{eqnarray}}
\newcommand{\eea}{\end{eqnarray}}

\numberwithin{equation}{section}
\begin{document}
\title[B\MakeLowercase{ohr-Type Inequalities with} F\MakeLowercase{r\'{e}chet Derivative and} A\MakeLowercase{rea Terms in} C\MakeLowercase{omplex Banach Spaces}]
{B\MakeLowercase{ohr-Type Inequalities with} F\MakeLowercase{r\'{e}chet Derivative and} A\MakeLowercase{rea Terms in} C\MakeLowercase{omplex Banach Spaces}}

\date{}
\author[N. S\MakeLowercase{arkar and} P. D\MakeLowercase{as}]{N\MakeLowercase{abadwip} S\MakeLowercase{arkar and} P\MakeLowercase{radip} D\MakeLowercase{as}}

\address{Amity School of Applied Sciences, Amity University Mumbai, Panvel, Navi Mumbai, Maharashtra-410206, India.}
\email{nsarkar@mum.amity.edu, nabadwipsarkar52@gmail.com}
\address{Department of Mathematics, Raiganj University, Raiganj, West Bengal-733134, India.}
\email{pradipsmath@gmail.com}

\renewcommand{\thefootnote}{}
\footnote{*\emph{Corresponding Author}: Pradip Das}
\renewcommand{\thefootnote}{\arabic{footnote}}
\setcounter{footnote}{0}

\begin{abstract}
We establish new Bohr-type inequalities for holomorphic mappings from the unit ball of a complex Banach space into the closed unit disk. Using Fr\'{e}chet derivatives, Schwarz mappings of prescribed orders and area functionals, we obtain sharp Bohr radii characterized by explicit equations. We further prove a refined Bohr inequality involving derivative, coefficient and area terms and show that the corresponding radius is the unique positive solution of $r^m=(\sqrt{17}-3)/4$. The sharpness of the obtained radii and constants is also established. Our results extend several recent Bohr-type inequalities for bounded analytic functions on the unit disk to the setting of holomorphic mappings on complex Banach spaces.
\end{abstract}

\thanks{Typeset by \AmS -\LaTeX}
\maketitle
\section{{\bf Introduction}}

Let $H^\infty$ denote the Banach space of all bounded analytic functions on the unit disk
$\mathbb{U}=\{z\in\mathbb{C}:|z|<1\},$
equipped with the norm $\|f\|_\infty=\sup_{z\in\mathbb{D}}|f(z)|.$ Also Let $\mathcal{B}:=\{f\in H^\infty(\mathbb{D}): \|f\|_\infty\le 1\},$
denote the class of all bounded analytic functions in the unit disk $\mathbb{D}$ with supremum norm at most one.\\

\medskip

A classical theorem of Bohr \cite{Bohr1914} states that if $f(z)=\sum_{n=0}^{\infty}a_n z^n \in H^\infty,$
then
\[
B_0(f,r):=|a_0|+\sum_{n=1}^{\infty}|a_n|r^n
\leq \|f\|_\infty
\]
for all $r\leq 1/6$. Shortly thereafter, Riesz, Schur, and Wiener independently improved the radius to $1/3$, proving that
\[
|a_0|+\sum_{n=1}^{\infty}|a_n|r^n \leq \|f\|_\infty
\]
whenever $r\leq 1/3$ and that the constant $1/3$ is sharp. This result is now known as the classical Bohr theorem, while the number $1/3$ is referred to as the \emph{Bohr radius}. The sharpness follows from the family of disk automorphisms
\[
\varphi_a(z)=\frac{a-z}{1-az}, \qquad 0\le a<1.
\]

Although Bohr's original paper already contained the essential ideas leading to the sharp radius, a variety of alternative proofs and extensions have subsequently appeared. Comprehensive surveys of the classical theory and its developments can be found in \cite{AAP2017,GMR2018}. It is worth noting that no extremal function in $H^\infty$ attains the Bohr radius exactly; rather, the radius $1/3$ arises as a limiting value of suitable extremal families (see \cite{AKP2019,GMR2018,KP2017}).

During the past two decades, Bohr's phenomenon has attracted considerable attention and has been investigated in a wide range of settings, including multidimensional complex analysis, harmonic and quasi-conformal mappings, operator-valued functions, Dirichlet series, and various subclasses of analytic functions. Numerous refinements and generalizations have been established; see, for example, \cite{AA2013,BD2018,BB2004,KS2022,LLP2023, LL2020} and the references therein.


In order to present our results in a clear and systematic manner, we first fix some
notation. For a point $z=(z_1,z_2,\ldots,z_n)\in\mathbb{C}^n$, we denote by $\|z\|_\infty:=\max_{1\le i\le n}|z_i|$
the supremum norm. The unit polydisc in $\mathbb{C}^n$ is then defined by
$\mathbb{U}^n:=\{z\in\mathbb{C}^n:\ \|z\|_\infty<1\}.$\par
Let $X$ and $Y$ be complex Banach spaces.
\begin{defi}
Let $k\in\mathbb{N}$. A mapping $P\colon X\to Y$ is called a \emph{homogeneous polynomial of degree $k$}
if there exists a $k$-linear mapping $u\colon X^k\to Y$ such that
\[
P(x)=u(x,\ldots,x), \qquad x\in X.
\]
\end{defi}

Note that if $P$ is a homogeneous polynomial of degree $k$, then
\[
P(\lambda x)=\lambda^k P(x), \qquad x\in X,\ \lambda\in\mathbb{C}.
\]
Throughout this paper, the degree of a homogeneous polynomial is indicated by a subscript.
That is, if $P_k$ is a homogeneous polynomial, then its degree is $k$.
Moreover, if $P_k$ is a $k$-homogeneous polynomial from $X$ into $Y$, then there exists a unique
symmetric $k$-linear mapping $u$ such that
\[
P_k(x)=u(x,\ldots,x), \qquad x\in X.
\]

Let $D\subset X$ be a domain and let $F\colon D\to Y$ be a holomorphic mapping.
For $z\in D$, denote by $D^kF(z)$ the $k$-th Fr\'{e}chet derivative of $F$ at $z$.
If $D$ contains the origin, then $F$ admits the expansion
\begin{equation}
F(z)=\sum_{k=0}^{\infty}\frac{1}{k!}\,D^kF(0)(z^k)
\label{eq:Taylor}
\end{equation}
in a neighbourhood of the origin.
Since $\frac{1}{k!}D^kF(0)(z^k)$ is a homogeneous polynomial of degree $k$, we shall use the notation
\[
P_k(z):=\frac{1}{k!}D^kF(0)(z^k)
\]
throughout this paper.
If $D$ is a bounded balanced domain in $X$ and $F(D)$ is bounded, then the series
\eqref{eq:Taylor} converges uniformly on $rD$ for each $r\in(0,1)$.

Let $F\colon D\to Y$ be holomorphic. For $k\in\mathbb{N}$, we say that $z=0$ is a zero of order $k$ of $F$
if
\[
F(0)=0,\quad DF(0)=0,\ \ldots,\ D^{k-1}F(0)=0,\quad \text{but } D^kF(0)\neq 0.
\]

Let $B_X$ and $B_Y$ denote the open unit balls of the Banach spaces $X$ and $Y$, respectively.
A holomorphic mapping $\mu\colon B_X\to B_Y$ with $\mu(0)=0$ is called a \emph{Schwarz mapping}.
If $\mu_k$ is a Schwarz mapping such that $z=0$ is a zero of order $k$ of $\mu$, then the following estimate holds
(see, for example, \cite[Lemma~6.1.28]{GK2003}):
\begin{equation}
\|\mu_k(z)\|_Y \le \|z\|_X^{\,k}, \qquad z\in B_X.
\label{eq:Schwarz}
\end{equation}

Let $L(X,\mathbb{C})$ denote the space of continuous linear operators from $X$ into $\mathbb{C}$.
For each $x\in X\setminus\{0\}$, define
\[
T(x):=\{\,\ell_x\in L(X,\mathbb{C}) : \ell_x(x)=\|x\|,\ \|\ell_x\|=1\,\}.
\]
By the Hahn--Banach theorem, the set $T(x)$ is nonempty.

\subsection{{\bf A Refinement of the Classical Bohr Inequality:}}
The classical Bohr phenomenon has been extensively investigated for various classes of analytic functions and in different functional settings. In particular, a natural question is whether the Bohr inequality involving the term $|f(z)|$ can be refined by replacing it with the stronger quantity $|f(z)|^2$. Addressing this problem, Liu et. al. \cite{LSX2018} established sharp Bohr-type inequalities for analytic functions in the unit disk involving higher-order derivatives. Their results not only provide a refinement of the classical inequality but also determine the corresponding optimal Bohr radii. The following theorems summarize their main findings.
\begin{theoA}\cite[Theorem 2.2]{LSX2018}
Suppose that $N\geq 2$ is an integer, and let $f(z)=\sum_{k=0}^{\infty} a_k z^k$ be analytic in $\mathbb{U}$ with $|f(z)|<1$ in $\mathbb{U}$. Then
\[
|f(z)|+\sum_{k=N}^{\infty}\left|\frac{f^{(k)}(z)}{k!}\right||z|^k\leq 1,
\]
for $|z|=r\leq R_N$, where $R_N$ is the smallest positive root of $\psi_N(r)=(1+r)(1-2r)(1-r)^{N-1}-2r^N=0.$
Moreover, the radius $R_N$ is best possible.
\end{theoA}

\begin{theoB}\cite[Corollary 2.3]{LSX2018}
Suppose that $f(z)=\sum_{k=0}^{\infty} a_k z^k$ is analytic in $\mathbb{U}$ and satisfies $|f(z)|<1$ in $\mathbb{U}$. Then
\[
|f(z)|^2+\sum_{k=N}^{\infty}\left|\frac{f^{(k)}(z)}{k!}\right|
|z|^k\leq 1,
\]
for $|z|=r\leq R_N^{\prime}$, where $R_N^{\prime}$ is the positive root of $(1+r)(1-2r)(1-r)^{N-1}-r^N=0.$ 

Moreover, the radius $R_N^{\prime}$ is best possible.
\end{theoB}
This result generates a significant amount of research activity on Bohr-Rogosinki inequalities for different classes of functions. However, Liu et al.~\cite[Theorem 7]{LLP2021} have proved the following result for functions in the class $\mathcal{B}$, where $|a_0|$ and $|a_1|$ are replaced by $|f(z)|$ and $|f'(z)|$, respectively.
\begin{theoC}
Suppose that $f\in\mathcal{B}$ and $f(z)=\sum_{n=0}^{\infty} a_n z^n.$
Then
\[
|f(z)|+|f'(z)|\,r+\sum_{n=2}^{\infty}|a_n|r^n+\left(\frac{1}{1+|a_0|}+\frac{r}{1-r}\right)\sum_{n=1}^{\infty}|a_n|^2r^{2n}\leq 1
\]
for $|z|=r\leq (\sqrt{17}-3)/4$, and the constant $(\sqrt{17}-3)/4$ is best possible.

Moreover,
\[
|f(z)|^2+|f'(z)|\,r+\sum_{n=2}^{\infty}|a_n|r^n+\left(\frac{1}{1+|a_0|}+\frac{r}{1-r}\right)\sum_{n=1}^{\infty}|a_n|^2r^{2n}\leq 1
\]

for $|z|=r\leq r_0$, where $r_0\approx 0.385795$ is the unique positive root of the equation
\[
1-2r-r^2-r^3-r^4=0.
\]
The constant $r_0$ is also best possible.
\end{theoC}
Recently, Ahamed and Roy \cite{AR2025} defined some functional as follows:
\begin{equation}
\begin{aligned}
\mathcal{D}_{f,p}(z,r)
:=\;& |f(z)|^p + |f'(z)|\,r
+\frac{|f''(z)|}{2!}\,r^2++\frac{|f'''(z)|}{3!}\,r^3
+\sum_{n=4}^{\infty}|a_n|r^n  \\
&+\frac{|a_1|^2 r^4}{1-r}
+\left(\frac{1}{1+|a_0|}+\frac{r}{1-r}\right)\sum_{n=2}^{\infty}|a_n|^2r^{2n}.
\end{aligned}
\end{equation}

In the same paper, Ahamed and Roy \cite{AR2025} obtained the following refined Bohr inequality.
\begin{theoD}\cite[Theorem 2.2.]{AR2025}
Let $f(z)=\sum_{n=0}^{\infty} a_n z^n $ be an analytic in $\mathbb{U}$ with $|f(z)|<1$ for $z\in \mathbb{U}$.
Then we have \[
\mathcal{D}_{f,1}(z,r)\leq 1\;\;\text{for}\;\;|z|=r\leq r_1\approx 0.285086,\]
 where $r_1$ is the unique root in $(0,1)$ of the equation $-1+4r-2r^2+3r^4+2r^5-2r^6-2r^7=0.$

Moreover, 
\[\mathcal{D}_{f,2}(z,r)\leq 1\;\;\text{for}\;\;|z|=r\leq r_2\approx 0.386055,\]
 where $r_2$ is the unique root in $(0,1)$ of the equation $-1+3r-r^2-r^3+2r^4+r^5-r^6-r^7=0.$

The equality $\lim_{a\to 1^-} \mathcal{D}_{f,p}(z,r)=1, p=1,2$ occurs for the function  $f_a$ defined as
\begin{equation}
f_a(z)
=\frac{a-z}{1-az}
= a-(1-a^2)\sum_{k=1}^{\infty} a^{k-1} z^k,
\qquad z\in\mathbb{D},
\quad a\in(0,1).
\end{equation}
\end{theoD}
 
It is natural to pose the following problem.

\begin{que1}
Do the refined Bohr inequalities of Theorems D admit multidimensional analogues for holomorphic mappings $F:B_X\to \mathbb{U}$
 in complex Banach spaces when combined with one or several Schwarz functions? 
 \end{que1}
 \subsection{\bf Improved Bohr inequality for the class $\mathcal{B}$:}

Let $f(z)=\sum_{n=0}^{\infty}a_nz^n, z\in\mathbb U,$ be a holomorphic function. For $0<r<1$, let $D_r={z\in\mathbb C:|z|<r},$ and define the associated area functional by $S_r=S_r(f):=\iint_{D_r}|f'(z)|^2,dA(z),$ where $dA$ denotes the planar Lebesgue area measure.

A sharp estimate for $S_r$ was established by Kayumov, Khammatova and Ponnusamy \cite{KKP2021}. More precisely, if $f\in\mathcal B$, then
\bea\label{Sr}
\frac{S_r}{\pi}:=\sum_{n=1}^{\infty}n|a_n|^2r^{2n}
\le
\frac{r^2(1-|a_0|^2)^2}
{(1-|a_0|^2r^2)^2},
\qquad
0<r\le\frac1{\sqrt2}.
\eea

By incorporating the area term $S_r$ into the classical Bohr inequality, Kayumov and Ponnusamy \cite{KP2018} obtained a substantial refinement of the Bohr phenomenon. Recently, Ahamed and Roy \cite{AR2025} further improved this result and proved the following sharp theorem.

\begin{theoE}\cite[Theorem 2.3]{AR2025}
Suppose that $f(z)=\sum_{n=0}^{\infty}a_nz^n$ is analytic in $\mathbb U$ and satisfies $|f(z)|<1$ for all $z\in\mathbb U$. Then
\[
\begin{aligned}
G_{1,f}(r)
:=&
|f(z)|
+|f'(z)|,r
+\sum_{n=2}^{\infty}|a_n|r^n
+\left(\frac{1}{1+|a_0|}
+\frac{r}{1-r}\right)
\sum_{n=1}^{\infty}|a_n|^2r^{2n}
+\lambda\frac{S_r}{\pi}
\le1
\end{aligned}
\]
for $|z|=r\le(\sqrt{17}-3)/4$, where $\lambda=\frac{221-43\sqrt{17}}{64}.$
Furthermore, both the radius $(\sqrt{17}-3)/4$ and the constant $\frac{221-43\sqrt{17}}{64}$
are sharp.
\end{theoE}

Theorem E provides a sharp refinement of the classical Bohr inequality involving derivative, coefficient-square and area contributions. Since many classical Bohr-type phenomena admit natural extensions to holomorphic mappings on Banach spaces, it is natural to ask whether Theorem E remains valid in this more general setting.

\medskip

\noindent
\begin{que2}
Can Theorem E be extended to holomorphic mappings on complex Banach spaces? If so, what are the corresponding sharp radius and the best possible constant $\lambda$?
\end{que2}
\medskip

The main purpose of this paper is to answer Questions 1.1 and 1.2 in the affirmative by establishing sharp Bohr-type inequalities for holomorphic mappings associated with Schwarz functions in Banach spaces.

\section{{\bf Auxiliary Lemmas.}} 
The following are key lemma of this paper and will be used to prove the main results.
\begin{lem}\label{L1} \cite{CHPV}
Suppose that $B_{X}$ and $B_{Y}$ are the unit balls of the complex Banach spaces $X$ and $Y$, respectively. Let $f: B_{X} \to B_{Y}$ be a holomorphic mapping. Then
\begin{equation*}
\|f(z)\|_{Y} \le \frac{\|f(0)\|_{Y} + \|z\|_{X}}{1 + \|f(0)\|_{Y}\|z\|_{X}}, \qquad z \in B_{X}.
\end{equation*}
This estimate is sharp with equality possible for each value of $\|f(0)\|_{Y}$ and for each $z \in B_{X}$.
\end{lem}
\begin{lem}\label{L20}\cite{DP2008}, \cite[Theorem 2]{R1985} Suppose that $f$ is an analytic self-maps of the unit disk $\mathbb{U}$.  Then for all $k=1, 2, 3,....$ we have 
\[|f^{(k)}(z)|\leq \frac{k!(1-|f(z)|^2)}{(1-|z|^2)^k}(1+|z|)^{k-1},\;|z|<1.\]
\end{lem}
\begin{lem}\label{L2}\cite{LLP2021}
Suppose that $f$ is an analytic self-maps of the unit disk $\mathbb{U}$. Then for any $N\in \mathbb{N}$, the  following inequality holds
\[\sum_{n=N}^{\infty} |a_n|\, r^{n}+\text{sgn}(t)\sum_{n=1}^t|a_n|^2\frac{r^N}{1-r}+ \left( \frac{1}{1+|a_0|}+ \frac{r}{1-r} \right)\sum_{n=t+1}^{\infty} |a_n|^{2} r^{2n}\le (1-|a_0|^{2})\,\frac{r^N}{1-r},\]
for $r\in [0,1)$, where $t=\lfloor \frac{N-1}{2}\rfloor$.
\end{lem}

\begin{lem}\label{L3}
Let $B_{X}$ be the unit ball of a complex Banach space $X$ also let $F$ be a holomorphic mapping from $B_{X}$ to $\ol{\mathbb{U}}$ with
\[
F(z)=a+\sum_{s=1}^{\infty}P_{s}(z), \qquad z\in B_{X},
\]
where $P_{s}(z)=\frac{1}{s!}D^{s}F(0)(z^{s})$ and   $\mu_k:B_{X}\to B_{X}$ are  Schwarz mappings having $z=0$ as a zero of order $k$. Then, for each $k\in\mathbb{N}$, the inequality
\beas\label{eq:lemma3}
&&\sum_{j=N}^{\infty} |P_j(\mu_k(z))| +\text{sgn}(t)\sum_{n=1}^t\frac{|P_n(\mu_k(z))|^2r^{k(N-2n)}}{1-r^k}
+\left(\frac{1}{1+|a|}+\frac{r^k}{1-r^k}\right)
\sum_{j=t+1}^{\infty} |P_j(\mu_k(z))|^2\nonumber\\
&\le& (1-|a|^2)\frac{r^{Nk}}{1-r^k}
\eeas
holds for $\|z\|=r\in[0,1)$, where $t=\lfloor \frac{N-1}{2}\rfloor$.
\end{lem}

\begin{proof}
Fix an arbitrary point $z\in B_X\setminus\{0\}$ and set $z_0=\frac{z}{\|z\|_X}\in \partial B_X.$

Define the auxiliary function
\[
f(\zeta)=F(\zeta z_0), \qquad \zeta\in \overline{\mathbb U}.
\]
Since $F$ is holomorphic on $B_X$, the function $f$ is holomorphic in $\mathbb U$ and admits the expansion
\[
f(\zeta)=a+\sum_{s=1}^{\infty}P_s(z_0)\zeta^s,
\qquad \zeta\in\mathbb U.
\]

By a classical coefficient estimate (see, for example, \cite[p.~35]{GK2003}), we have
\begin{equation}\label{l0}
|P_s(z_0)|\le 1-|a|^2,\qquad s\ge1.
\end{equation}
Furthermore, an application of Lemma~\ref{L2} yields
\begin{align}
&\sum_{s=N}^{\infty}|P_s(z_0)|\,|\zeta|^s+\text{sgn}(t)\sum_{n=1}^t|P_n(z_0)|^2\frac{|\zeta|^N}{1-|\zeta|}
+\left(\frac{1}{1+|a|}+\frac{|\zeta|}{1-|\zeta|}\right)
\sum_{s=t+1}^{\infty}|P_s(z_0)|^2|\zeta|^{2s}
\nonumber\\
&\le
(1-|a|^2)\frac{|\zeta|^N}{1-|\zeta|}.
\label{l1}
\end{align}

Taking $\zeta=\|z\|_X=r<1$, we obtain from \eqref{l1} that
\begin{align}
&\sum_{s=N}^{\infty}|P_s(z)|
+\text{sgn}(t)\sum_{n=1}^t|P_n(z)|^2\frac{\|z\|_X^{N-2n}}{1-\|z\|_X}
+\left(\frac{1}{1+|a|}
+\frac{\|z\|_X}{1-\|z\|_X}\right)
\sum_{s=t+1}^{\infty}|P_s(z)|^2
\nonumber\\
&\le
(1-|a|^2)\frac{\|z\|_X^N}{1-\|z\|_X}.
\label{hhh1}
\end{align}

Next, consider the function
\[
\phi_1(x)=\frac{x}{1-x}, \;\text{and}\;\phi_2(x)=\frac{x^{N-2n}}{1-x}, \;\text{for}\;N-2n\geq 0.
\]
Since
\[
\phi_1'(x)=\frac{1}{(1-x)^2}>0\;\text{and}\; \phi_2(x)=\frac{(N-2n)x^{N-2n-1}}{1-x}+\frac{x^{N-2n}}{(1-x)^2}\geq 0 \;\text{for}\;0\leq x<1.
\]
the function $\phi_i(i=1,2)$ are strictly increasing on $(0,1)$. As $r^k\le r$ for every $k\ge1$, it follows that
\[
\phi_i(r^k)\le \phi_i(r),\; \;i=1,2
\]
or equivalently,
\[
\frac{r^k}{1-r^k}
\le
\frac{r}{1-r}
=
\frac{\|z\|_X}{1-\|z\|_X}\;\text{and}\; \frac{r^{k(N-2n)}}{1-r^k}\leq \frac{r^{N-2n}}{1-r}=\frac{\|z\|_X^{N-2n}}{1-\|z\|_X}.
\]

Consequently,
\begin{align}
&\sum_{s=N}^{\infty}|P_s(z)|
+\text{sgn}(t)\sum_{n=1}^t|P_n(z)|^2\frac{r^{k(N-2n)}}{1-r^k}
+\left(\frac{1}{1+|a|}
+\frac{r^k}{1-r^k}\right)
\sum_{s=t+1}^{\infty}|P_s(z)|^2
\nonumber\\
&\le
\sum_{s=N}^{\infty}|P_s(z)|
+\text{sgn}(t)\sum_{n=1}^t|P_n(z)|^2\frac{\|z\|_X^{N-2n}}{1-\|z\|_X}
+\left(\frac{1}{1+|a|}
+\frac{\|z\|_X}{1-\|z\|_X}\right)
\sum_{s=t+1}^{\infty}|P_s(z)|^2.
\label{aq1}
\end{align}

Finally, combining \eqref{eq:Schwarz}, \eqref{hhh1}, and \eqref{aq1}, we deduce that
\begin{align*}
&\sum_{s=N}^{\infty}|P_s(\mu_k(z))|
+\text{sgn}(t)\sum_{n=1}^t|P_n(z)|^2\frac{r^{k(N-2n)}}{1-r^k}
+\left(\frac{1}{1+|a|}
+\frac{r^k}{1-r^k}\right)
\sum_{s=t+1}^{\infty}|P_s(\mu_k(z))|^2
\\
&\le
(1-|a|^2)
\frac{\|\mu_k(z)\|_X^N}
     {1-\|\mu_k(z)\|_X}
\le
(1-|a|^2)\frac{r^{Nk}}{1-r^k}.
\end{align*}
This completes the proof.
\end{proof}
\section{{\bf Main results and their proofs.}}
In the setting of holomorphic mappings in Banach spaces, the Bohr phenomenon exhibits several interesting features when combined with Schwarz functions. The following result provides a sharp refined version of the Bohr inequality involving several Schwarz functions and establishes the corresponding best possible Bohr radius.

Throughout this paper, for $p\in \{1,2\}$, we consider the following functionals, which play a fundamental role in our investigation of improved Bohr-type inequalities:

\bea\label{T1.1}\mathcal{J}_{f,p}(z,r):&=&|F(\mu_m(z))|^p+|D F(\mu_m(z))\mu_m(z)|+\frac{1}{2!}|D^2 F(\mu_m(z)) ((\mu_m(z))^2)|\nonumber\\
&&+\frac{1}{3!}|D^3 F(\mu_m(z)) ((\mu_m(z))^3)|+\sum_{j=4}^{\infty} |P_j(\mu_k(z))| +\frac{|P_1(\mu_k(z))|^2r^{2k}}{1-r^k}\nonumber\\
&&+\left(\frac{1}{1+|a|}+\frac{r^k}{1-r^k}\right)
\sum_{j=2}^{\infty} |P_j(\mu_k(z))|^2,
\eea

\begin{theo}\label{T1}
Let $B_{X}$ be the unit ball of a complex Banach space $X$ also let $F$ be a holomorphic mapping from $B_{X}$ to $\ol{\mathbb{U}}$ with
\[
F(z)=a+\sum_{s=1}^{\infty}P_{s}(z), \qquad z\in B_{X},
\]
where $P_{s}(z)=\frac{1}{s!}D^{s}F(0)(z^{s})$ and   $\mu_k,\mu_m:B_{X}\to B_{X}$ are  Schwarz mappings having $z=0$ as a zero of order $k,m$ respectively. Then,  for each $k,m\in\mathbb{N}$, the inequality
\[\mathcal{J}_{f,1}(z,r)\leq 1\]
for $\|z\|=r\leq R_{1,m,k}<R_{1,m}$, where $R_{1,m}\in(0,1)$ is the unique root of the equation 
\[\frac{r^m(1+r^{3m})}{(1-r^{2m})^2}-\frac{1-r^m}{2}=0\]
and $R_{1,m,k}\in(0,1)$ is the smallest positive root of the equation
\bea\label{rr1}
-(1-r^m)+\frac{2r^m(1+r^{3m})}{(1-r^{2m})^2}+\frac{2r^{4k}}{1-r^k}=0.
\eea

Moreover 
 \[\mathcal{J}_{f,2}(z,r)\leq 1\]
 for $\|z\|=r\leq R_{2,m,k}<R_{2,m}$, where $R_{2,m}\in(0,1)$ is the unique root of the equation 
 \[\frac{r^m(1+r^m)(1+r^{3m})}{(1-r^{2m})^2}-(1-r^{2m})=0\]
and $R_{2,m,k}\in(0,1)$ is the smallest positive root of the equation
\bea\label{rr2}-(1-r^{2m})+\frac{r^m(1+r^m)(1+r^{3m})}{(1-r^{2m})^2}+\frac{r^{4k}(1+r^m)^2}{1-r^k}=0.\eea
The equality $\lim_{a\to 1^-} \mathcal{J}_{F,p}(z,r)=1, p=1,2$ occurs for the function  $F$ defined in (\ref{F1}).
\end{theo}
\begin{rem}
Theorem \ref{T1} may be viewed as a Banach-space-valued extension and refinement of Theorem D\cite[Theorem 2.2]{AR2025}. Indeed, when $X=\mathbb{C}$, $m=k=1$, and the Schwarz mappings are chosen as $\mu_1(z)=\mu_2(z)=z,$ the holomorphic mapping $F(z)=a+\sum_{s=1}^{\infty}P_s(z)$ reduces to an analytic function $f(z)=\sum_{n=0}^{\infty}a_n z^n$ on the unit disk. In this setting, the functionals $\mathcal{J}_{F,p}$ become natural analogues of the functionals $\mathcal{D}_{f,p}$ considered in \cite{AR2025}. Consequently, Theorem \ref{T1} extends the corresponding Bohr-type inequalities from scalar-valued analytic functions on $\mathbb{D}$ to holomorphic mappings defined on the unit ball of an arbitrary complex Banach space.\par

Furthermore, the presence of the parameters $m$ and $k$, arising from the Schwarz mappings $\mu_m$ and $\mu_k$, provides additional flexibility and yields a broader family of Bohr-type radii. Thus, Theorem \ref{T1} not only generalizes the result of \cite{AR2025} but also establishes a unified framework encompassing both the classical one-variable case and the Banach-space setting.
\end{rem}

Table \ref{tab3.1} and Figures \ref{f31}--\ref{f32} present the numerical values and graphical locations of the radii $R_{1,m,k}$ and $R_{2,m,k}$. The values listed in Table \ref{tab3.1} correspond to the smallest roots of equations \eqref{rr1} and \eqref{rr2}, respectively, in the interval $(0,1)$. Figures \ref{f31} and \ref{f32} illustrate the locations of these roots for selected values of $m$ and $k$, confirming the numerical results.
 
\begin{table}[H]
\centering
\begin{minipage}[b]{0.45\textwidth}
\centering
\begin{tabular}{|c|c|c|c|}
\hline
\textbf{k} & \textbf{m} &\textbf{$R_{1,m,k}$} & \textbf{$R_{2,m,k}$} \\ 
\hline
1 & 1 & 0.286200& 0.386050 \\ 
\hline
2 & 2 & 0.534970& 0.621330 \\ 
\hline
2 & 3 & 0.644090 & 0.706450\\ 
\hline
3 & 4 & 0.724270& 0.777320 \\ 
\hline
\end{tabular}
\caption{Numerical values of $R_{1,m,k}$ and $R_{2,m,k}$ for selected values of $m$ and $k$.}
\label{tab3.1}
\end{minipage}%
\hfill
\begin{minipage}[b]{0.55\textwidth}
\centering
\begin{figure}[H]
\centering
\includegraphics[width=\textwidth]{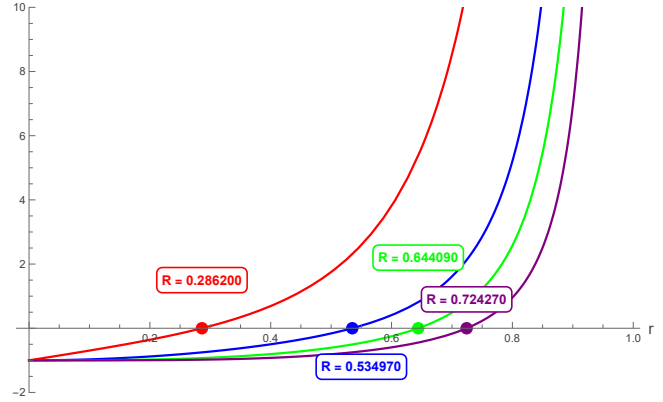}
\caption{The graphs exhibit the locations of the roots $R_{1,m,k}$ in $(0,1)$ for different values of $k,m$.}
\label{f31}
\end{figure}
\end{minipage}
\begin{minipage}[b]{0.55\textwidth}
\centering
\begin{figure}[H]
\centering
\includegraphics[width=\textwidth]{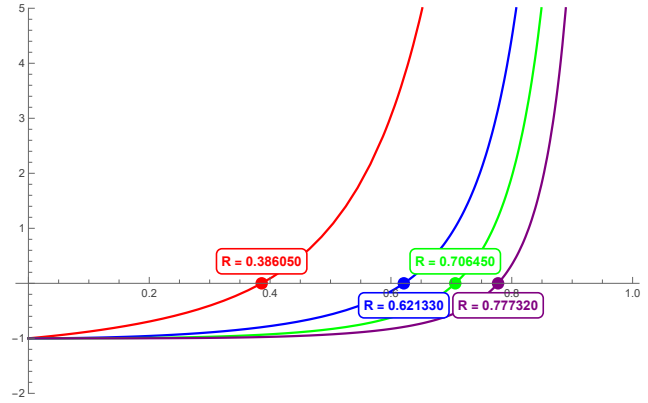}
\caption{The graphs exhibit the locations of the roots $R_{2,m,k}$ in $(0,1)$ for different value of $m$.}
\label{f32}
\end{figure}
\end{minipage}

\end{table}

\begin{proof}[{\bf Proof of Theorem \ref{T1.1}}]
Let $z\in B_X\setminus\{0\}$ be fixed and define $z_0=\frac{z}{\|z\|_X}\in \partial B_X.$
We consider the holomorphic function $f: \mathbb{U} \to \overline{\mathbb{U}}$ defined by
\[
f(\zeta) = F(\zeta z_0), \qquad \zeta \in \mathbb{U}.
\]
 Then 
\begin{equation}\label{eq:fexp}
f(\zeta)=a_j+\sum_{s=1}^{\infty}(P_s)(z_0)\zeta^s,
\qquad \zeta\in\mathbb{U}.
\end{equation}
Hence, we get (see, e.g., \cite[p.~35]{GK2003}),
$|(P_s)(z_0)|\le 1-|a|^2$.

Differentiating $k-$ times  $f(\zeta)= F(\zeta z_0)$, we have 
\[ f^{(k)}(\zeta)=D^k F(\zeta z_0) (z_0^k).\]
Then by Schwarz-Pick Lemma, we can conclude that 
\[ |D^kF(\zeta z_0)(z_0)^k|\leq \frac{k!(1-|f(\zeta)|^2)}{(1-|\zeta|^2)^k}(1+|\zeta|)^{k-1}.\]

Then by Schwarz-Pick Lemma, we can deduce that
\[ |D^kF(\zeta z_0)(z_0^k)|=|f^{(k)}(\zeta)|\leq \frac{k!(1-|f(\zeta)|^2)}{(1-|\zeta|^2)^k}(1+|\zeta|)^{k-1}.\]
Setting $\zeta=\|z\|_X=r<1$ then above equation Eqs. become 
\bea\label{eq11} |D^kF(z) (z^k)|\leq \frac{k!(1-|F(z)|^2)}{(1-\|z\|^2)^k}(1+\|z\|)^{k-1}\|z\|^k.\eea
Let $t_1(x)=\frac{(1+x)^{k-1}x^k}{(1-x^2)^k}$  and $t_2(x)=\frac{a+x}{1+ax}$ for  $0\le x\le x_0<1,$ where $k\ge 1$. A straightforward computation yields $t_1'(x)=\frac{x^{k-1}\bigl(k+(2k-1)x+x^2\bigr)}{(1-x)^{k+1}(1+x)^2}$ and $ t_2(x)=\frac{1-a^2}{(1+ax)^2}$.

Since $k\ge 1$ and $0\le x<1$, it follows that $t_1'(x)\ge 0$ and $t_2'(x)\geq 0$.
Hence $t_1(x)$ and $t_2(x)$ are  increasing functions on $[0,1]$. Since $\|\mu_m(z)\|_X\le \|z\|_X^{\,m}\leq 1$, then we get 
\[t_1(\|\mu_m(z)\|_X)\leq t_1(\|z\|_X^{\,m})\]
i.e., 
\[
\frac{\bigl(1+\|\mu_m(z)\|_X\bigr)^{k-1}}
     {\bigl(1-\|\mu_m(z)\|_X^2\bigr)^k}||\mu_m(z)|^k
\le
\frac{\bigl(1+\|z\|_X^{\,m}\bigr)^{k-1}}
     {\bigl(1-\|z\|_X^{\,2m}\bigr)^k}\|z\|_X^{\,km}\;\text{for all}\;k\geq 1
\]
and 
\[t_2(\|\mu_m(z)\|_X)\leq t_2(\|z\|_X^{\,m}) \]
i.e., 
\[|F(\mu_m(z))|\leq \frac{a+\|\mu_m(z)\|_X}{1+a\|\mu_m(z)\|_X}\leq \frac{a+\|z\|_X^m}{1+a\|z\|_X^m}.\]
Now from (\ref{eq11}), we deduce that 
\bea\label{eq12}|D^kF(\mu_m(z)) ((\mu_m(z))^k)|&\leq&  \frac{k!(1-|F(\mu_m(z)|^2)\bigl(1+|\mu_m(z)|\bigr)^{k-1}}{\bigl(1-|\mu_m(z)|^2\bigr)^k}\nonumber\\
&\leq& \frac{k!(1-|F(\mu_m(z)|^2)\bigl(1+\|z\|_X^{\,m}\bigr)^{k-1}}{\bigl(1-\|z\|_X^{\,2m}\bigr)^k}\|z\|_X^{\,km}.\eea

We now divide the proof into two cases according to the values of $(p)$.

{\bf Case 1.} Suppose that  $p=1$.
Let $x=|a|\in[0,1]$. If $x=1$, then $P_s(z_0)=0$ for all $s\ge1$, and
\eqref{T1.1} holds trivially. 

Hence, we assume $x\in[0,1)$. Using  \eqref{eq12}, Lemma \ref{L1} and \ref{L3}, we obtain
\begin{align}
\mathcal{J}_{F,1}(z,r):=&|F(\mu_m(z))|+|D F(\mu_m(z))\mu_m(z)|+\frac{1}{2!}|D^2 F(\mu_m(z)) ((\mu_m(z))^2)|\nonumber\\
&+\frac{1}{3!}|D^3 F(\mu_m(z)) ((\mu_m(z))^3)|+\sum_{j=4}^{\infty} |P_j(\mu_k(z))| +\frac{|P_1(\mu_k(z))|^2r^{2k}}{1-r^k}
\nonumber\\
&+\left(\frac{1}{1+|a|}+\frac{r^k}{1-r^k}\right)
\sum_{j=2}^{\infty} |P_j(\mu_k(z))|^2\nonumber\\ 
\le & |F(\mu_m(z))|+\frac{r^m}{1-r^{2m}}(1-|F(\mu_m(z))|^2)+\frac{r^{2m}(1+r^m)}{(1-r^{2m})^2}(1-|F(\mu_m(z))|^2)\nonumber\\
&\frac{r^{3m}(1+r^m)^2(1-|F(\mu_m(z))|^2)}{(1-r^{2m})^3}+(1-x^2)\frac{r^{4k}}{1-r^k} \nonumber\\
=& \frac{x+r^m}{1+xr^m}+\left(\frac{r^m}{1-r^{2m}}+\frac{r^{2m}(1+r^m)}{(1-r^{2m})^2}+\frac{r^{3m}(1+r^m)^2}{(1-r^{2m})^3}\right)\bigg(1-\left(\frac{x+r^m}{1+xr^m}\right)^2\bigg)\nonumber\\
&+(1-x^2)\frac{r^{4k}}{1-r^k}\nonumber\\
=& \frac{x+r^m}{1+xr^m}+\frac{r^m(1+r^m)(1+r^{3m})(1-x^2)}{(1-r^{2m})^2(1+xr^m)^2}+(1-x^2)\frac{r^{4k}}{1-r^k}\nonumber\\
=&1+\frac{(1-x)G_1(x,r)}{1+xr^m}.\label{eq:G1}
\end{align}
where 
\[
G_1(x,r)=-(1-r^m)+\frac{r^m(1+r^m)(1+r^{3m})(1+x)}{(1-r^{2m})^2(1+xr^m)}+(1+x)\frac{r^{4k}}{1-r^k}.
\]
It is worth pointing out that, in the third inequality of Eq. (\ref{eq:G1}), we have used the fact that 
\bea\label{kkkk1} \theta (t)=t+\lambda (1-t^2)\leq \theta(t_0),\eea
whenever \[t=F(\mu_m(z))\leq t_0=\frac{x+r^m}{1+xr^m} \;\text{and}\; \lambda=\frac{r^m(1+r^m)(1+r^{3m})}{(1-r^{2m})^3}.\]
Inequality (\ref{kkkk1}) holds provided that $\frac{r^m(1+r^m)(1+r^{3m})}{(1-r^{2m})^3}\leq \frac{1}{2}$, which is satisfied for $0\leq r\leq R_{1,m}$, where $R_{1,m}$ is the unique positive root of the equation 
\bea\label{R1m}\frac{r^m(1+r^{3m})}{(1-r^{2m})^2}-\frac{1-r^m}{2}=0.\eea

Now differentiating $G_1$ partially with respect to $x$, we get
\[
\frac{\partial G_1(x,r)}{\partial x}=\frac{r^m(1+r^m)(1+r^{3m})(1-r^m)}{(1-r^{2m})^2(1+xr^m)^2}+\frac{r^{4k}}{1-r^k}\geq 0.
\]
Therefore $G_1(x,r)$ is a monotonically increasing function of $x\in [0,1)$ and hence, we gave 
\[G_1(x,r)\leq G_1(1,r)=-(1-r^m)+\frac{2r^m(1+r^{3m})}{(1-r^{2m})^2}+\frac{2r^{4k}}{1-r^k}:=G(r).\]

Now differentiating $G(r)$ with respect to $r$, we get 
\[G(r)=mr^{m-1}+\frac{2mr^{m-1}\left(1+4r^{2m}+2r^{3m}+2r^{5m}+3r^{6m}\right)}{(1-r^{2m})^3}+\frac{2k r^{4k-1}(4-3r^k)}{(1-r^k)^2}\geq 0\]
Hence, $G$ is increasing on $[0,1)$. Moreover,
\[
G(0)=-1
\quad\text{and}\quad
\lim_{r\to 1^-}G(r)=+\infty.
\]
Therefore, there exists a unique root $R_{1,m,k}\in(0,1)$ of the equation $G(r)=0$. Consequently,
\[
\mathcal{J}_{f,1}(z)\le 1\]
for all $r\le R_{1,m,k}$.

We now claim that $R_{1,m,k}<R_{1,m}$. Suppose, on the contrary that $R_{1,m,k}\geq R_{1,m}$. Then for any $r\geq R_{1,m}$, we have $\frac{r^m(1+r^{3m})}{(1-r^{2m})^2}-\frac{1-r^m}{2}\geq  0$ and hence 
\beas -(1-r^m)+\frac{2r^m(1+r^{3m})}{(1-r^{2m})^2}+\frac{2r^{4k}}{1-r^k}>0.\eeas
This contradicts $R_{1,m,k}$ roots of (\ref{rr1}). Therefore $R_{1,m,k}< R_{1,m}$.

\medskip

{\bf Case 2.} For $p=2$. Assume that $|a|=x\in[0,1]$. If $x=1$, then $P_s(z_0)=0$ for all $s\ge1$, and
\eqref{T1.1} holds trivially. 

Hence, we assume $x\in[0,1)$. Using  \eqref{eq12}, Lemma \ref{L1} and \ref{L3}, we obtain
\begin{align}
\mathcal{J}_{F,2}(z,r):=&|F(\mu_m(z))|^2+|D F(\mu_m(z))\mu_m(z)|+\frac{1}{2!}|D^2 F(\mu_m(z)) ((\mu_m(z))^2)|\nonumber\\
&+\frac{1}{3!}|D^3 F(\mu_m(z)) ((\mu_m(z))^3)|+\sum_{j=4}^{\infty} |P_j(\mu_k(z))| +\frac{|P_1(\mu_k(z))|^2r^{2k}}{1-r^k}
\nonumber\\
&+\left(\frac{1}{1+|a|}+\frac{r^k}{1-r^k}\right)
\sum_{j=2}^{\infty} |P_j(\mu_k(z))|^2\nonumber\\ 
\le & |F(\mu_m(z))|^2+\frac{r^m}{1-r^{2m}}(1-|F(\mu_m(z))|^2)+\frac{r^{2m}(1+r^m)}{(1-r^{2m})^2}(1-|F(\mu_m(z))|^2)\nonumber\\
&+\frac{r^{3m}(1+r^m)^2(1-|F(\mu_m(z))|^2)}{(1-r^{2m})^3}+(1-x^2)\frac{r^{4k}}{1-r^k} \nonumber\\
=& \left(\frac{x+r^m}{1+xr^m}\right)^2+\left(\frac{r^m}{1-r^{2m}}+\frac{r^{2m}(1+r^m)}{(1-r^{2m})^2}+\frac{r^{3m}(1+r^m)^2}{(1-r^{2m})^3}\right)\bigg(1-\left(\frac{x+r^m}{1+xr^m}\right)^2\bigg)\nonumber\\
&+(1-x^2)\frac{r^{4k}}{1-r^k}\nonumber\\
=& \left(\frac{x+r^m}{1+xr^m}\right)^2+\frac{r^m(1+r^m)(1+r^{3m})(1-x^2)}{(1-r^{2m})^2(1+xr^m)^2}+(1-x^2)\frac{r^{4k}}{1-r^k}\nonumber\\
=&1+\frac{(1-x^2)G_1(x,r)}{(1+xr^m)^2},\label{eq:G2}
\end{align}
where $G_2(x,r)=-(1-r^{2m})+\frac{r^m(1+r^m)(1+r^{3m})}{(1-r^{2m})^2}+\frac{r^{4k}(1+xr^m)^2}{1-r^k}$.\\

The third inequality of Eq. (\ref{eq:G2}) holds for $r\in [0,1]$ satisfying $\frac{r^m(1+r^m)(1+r^{3m})}{(1-r^{2m})^3}\leq 1$, which is satisfied for $0\leq r\leq R_{2,m}$, where $R_{2,m}$ is the unique positive root of the equation 
\[\frac{r^m(1+r^m)(1+r^{3m})}{(1-r^{2m})^2}-(1-r^{2m})=0.\]

Differentiating $G_2(x,r)$ with respect to $x$, we obtain
\[
\frac{\partial G_2(x,r)}{\partial x}
=\frac{2r^{4k+m}(1+xr^m)}{1-r^k}\ge 0.
\]
Consequently, $G_2(x,r)$ is increasing in $x\in[0,1)$. Therefore,
\[
G_2(x,r)\le G_2(1,r)
=-(1-r^{2m})+\frac{r^m(1+r^m)(1+r^{3m})}{(1-r^{2m})^2}+\frac{r^{4k}(1+r^m)^2}{1-r^k}:=H(r).
\]
Now differentiating $G(r)$ with respect to $r$, we get 
\beas
H'(r)
&&=2m r^{2m-1}+
\frac{m r^{m-1}+2m r^{2m-1}+4m r^{4m-1}+5m r^{5m-1}}
     {(1-r^{2m})^2} \\
&&\quad+
\frac{4m\bigl(r^{3m-1}+r^{4m-1}+r^{6m-1}+r^{7m-1}\bigr)}
     {(1-r^{2m})^3} \\
&&\quad+
\frac{
\Bigl(4k r^{4k-1}(1+r^m)^2
+2m r^{4k+m-1}(1+r^m)\Bigr)(1-r^k)
+k r^{5k-1}(1+r^m)^2
}
{(1-r^k)^2}
\ge 0.
\eeas
Hence, $H$ is increasing on $[0,1)$. Moreover,
\[
H(0)=-1
\quad\text{and}\quad
\lim_{r\to 1^-}H(r)=+\infty.
\]
Therefore, there exists a unique root $R_{2,m,k}\in(0,1)$ of the equation $H(r)=0$. Consequently,
\[
\mathcal{J}_{f,2}(z,r)\le 1\]
for all $r\le R_{2,m,k}$.

We now claim that $R_{2,m,k}<R_{2,m}$. Suppose, on the contrary that $R_{2,m,k}\geq R_{2,m}$. Then for any $r\geq R_{2,m}$, we have $\frac{r^m(1+r^m)(1+r^{3m})}{(1-r^{2m})^2}-(1-r^{2m})\geq  0$ and hence 
\beas -(1-r^{2m})+\frac{r^m(1+r^m)(1+r^{3m})}{(1-r^{2m})^2}+\frac{r^{4k}(1+r^m)^2}{1-r^k}>0.\eeas
This contradicts $R_{2,m,k}$ roots of (\ref{rr2}). Therefore $R_{2,m,k} < R_{2,m}$.

\medskip
The second part of the proof is to show that equality $\mathcal{J}_{f,p}(z,r)=1$ occurs.
Fix $z_0\in\partial B_X$. For $b\in(0,1)$, define
\begin{equation}\label{F1}
F_1(z)=f(l_{z_0}(z)),\qquad z\in B_X,
\end{equation}
where
\[
f(\zeta)=\frac{b-\zeta}{1-b\zeta},\qquad \zeta\in\mathbb{U},
\]
and $l_{z_0}\in T(z_0)$. Let $\mu_m(z)=l_{z_0}(z)^{m-1}z, m\geq 1.$ Then, for $r\in(0,1)$,
\[
F_1(rz_0)=\frac{b-r}{1-br}
=b+\sum_{s=1}^{\infty}P_s(rz_0),
\]
where $P_s(z_0)=-(1-b^2)b^{s-1}, s\ge1.$ Consequently, $P_s(\mu_m(rz_0))=-(1-b^2)b^{s-1}r^{ms}.$

Moreover,
\[
DF_1(z)(z^2)
=\frac{(1-b^2)l_{z_0}(z)}
{\bigl(1-bl_{z_0}(z)\bigr)^2}, \;\text{and}\; D^2F_1(z)(z^2)
=\frac{2b(1-b^2)[l_{z_0}(z)]^2}
{\bigl(1-bl_{z_0}(z)\bigr)^3}.
\]
and
\[DF_1(z)(z^3)=-\frac{6b^2(1-b^2)r^3}{(1-br)^4}.\]
Thus for $z=rz_0$, we have 
\begin{align*}
\mathcal{J}_{F_1,p}(z,r):=&|F_1(\mu_m(rz_0))|^p+|D F_1(\mu_m(rz_0))\mu_m(rz_0)|+\frac{1}{2!}|D^2 F_1(\mu_m(rz_0)) ((\mu_m(rz_0))^2)|\\
&+\frac{1}{3!}|D^3 F_1(\mu_m(rz_0)) ((\mu_m(rz_0))^3)|+\sum_{j=4}^{\infty}|P_j(\mu_k(rz_0))|+\frac{|P_1(\mu_k(rz_0))|^2r^{k}}{1-r^k}\\
&+\left(\frac{1}{1+|a|}+\frac{r^k}{1-r^k}\right)\sum_{j=2}^{\infty}|P_j(\mu_k(rz_0))|^2 \\
=&\left(\frac{b-r^m}{1-br^m}\right)^p+\frac{(1-b^2)r^m}{(1-br^m)^2}+\frac{b(1-b^2)r^{2m}}{(1-br^m)^3}+\frac{b^2(1-b^2)r^{3m}}{(1-br^m)^4}\\
&+\frac{(1-b^2)b^3r^{4k}}{1-br^k}+\left(\frac{1}{1+|b|}+\frac{r^k}{1-r^k}\right)\frac{(1-b^2)^2\,r^{2k}}{1-b^2r^{2k}}\\
&=1+\frac{(1-b^2)D_{p}(b,r)}{(1-br^m)^2}
\end{align*}
where
\[
\begin{aligned}
D_1(b,r)
=&-(1+r^m)(1-br^m)
+r^m+\frac{br^{2m}}{1-br^m}
+\frac{b^2r^{3m}}{(1-br^m)^2}+\frac{b^3r^{4k}(1-br^m)^2}
{(1-br^k)(1-b^2)}
\\
&+(1-b^2)
\left(\frac{1}{1+b}+\frac{r^k}{1-r^k}\right)
\frac{r^{2k}(1-br^m)^2}
{1-b^2r^{2k}}.
\end{aligned}
\]
and 
\[
\begin{aligned}
D_2(b,r)
=&-(1-r^{2m})
+r^m+\frac{br^{2m}}{1-br^m}
+\frac{b^2r^{3m}}{(1-br^m)^2}
+\frac{b^3r^{4k}(1-br^m)^2}
{(1-br^k)(1-b^2)}\\
&+(1-b^2)
\left(\frac{1}{1+b}+\frac{r^k}{1-r^k}\right)
\frac{r^{2k}(1-br^m)^2}
{1-b^2r^{2k}}.
\end{aligned}
\]
Letting $b\to1^-$, we obtain
\beas
&&\lim_{b\to1^-}\mathcal{J}_{F_2,p}(z,r)=1.\eeas

\end{proof}

A major difficulty in addressing Question 2.2 is that the area estimate (\ref{Sr}) is intrinsically one-dimensional, relying on the planar geometry of $\mathbb C$ and Parseval's identity. Consequently, no direct analogue of the classical area functional is available in the Banach-space setting. To overcome this difficulty, we introduce the following area-type functional:
\[
\frac{\mathcal S_r(f)}{\pi}
:=
\frac{r^2(1-|a_0|^2)^2}
{(1-|a_0|^2r^2)^2}.
\]
This quantity preserves the extremal growth behaviour of the classical area term and therefore serves as a natural substitute for $S_r$ in the Banach-space framework.

The next theorem answers the above question affirmatively and may be viewed as a Banach-space analogue of Theorem E.

\begin{theo}\label{T3}
Let $X$ be a complex Banach space and let $F(z)=a_0+\sum_{n=1}^{\infty}P_n(z)$ be a holomorphic mapping from $B_X$ into $\ol{\mathbb U}$ with $|F(z)|\leq 1$ for $z\in B_X$, wherewhere $P_{s}(z)=\frac{1}{s!}D^{s}F(0)(z^{s})$ and $\mu_m:B_{X}\to B_{X}$ are  Schwarz mappings having $z=0$ as a zero of order $m$.

Then
\[
\begin{aligned}
\mathcal G_{1,F}(r)
:=&
|F(\mu_m(z))|
+|DF(\mu_m(z))(\mu_m(z))|
+\sum_{n=2}^{\infty}|P_n(\mu_m(z))|
\\
&\quad
+\left(\frac{1}{1+|a_0|}
+\frac{r^m}{1-r^m}\right)
\sum_{n=1}^{\infty}|P_n(\mu_m(z))|^2
+\lambda\frac{\mathcal S_{|\mu_m(z)|}(F)}{\pi}
\le1
\end{aligned}
\]
for $|z|=r\le R_m$, where $R_m$ is the unique positive solution of $r^m=\frac{\sqrt{17}-3}{4}$ and $\lambda=\frac{221-43\sqrt{17}}{64}.$

Furthermore, both $R_m$ and the constant $\frac{221-43\sqrt{17}}{64}$
are sharp.
\end{theo}

\begin{rem}
When $X=\mathbb C$ and $m=1$, we have $\mu_1(z)=z$ and $R_1=\frac{\sqrt{17}-3}{4}.$
Therefore, Theorem E is recovered as a special case of the above theorem.
\end{rem}
Figure \ref{f33} illustrates the locations of these roots for selected values of $m$, confirming the numerical results.
 
\centering
\begin{figure}[H]
\centering
\includegraphics[width=\textwidth]{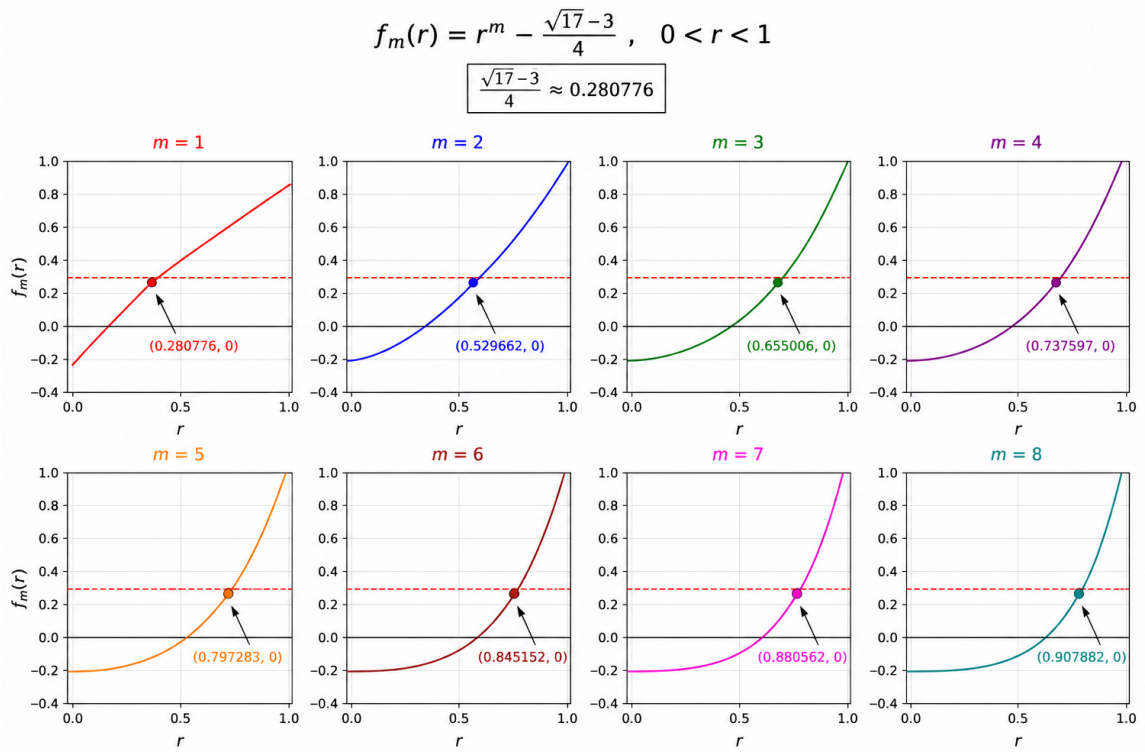}
\caption{The graphs exhibit the locations of the roots $R_{m}$ in $(0,1)$ for different value of $m$.}
\label{f33}
\end{figure}
\begin{proof}[{\bf Proof of Theorem \ref{T3}}]  Let $|a_0|=a\in(0,1)$. Combining Lemma \ref{L3} (with $N=2$) and (\ref{eq12}) together with the classical estimate for $|F(z)|$, we obtain
\[
\begin{aligned}
\mathcal{G}_{1,f}(r)
&:= |F(\mu_m(z))|+|DF(\mu_m(z))(\mu_m(z))| +\sum_{n=2}^{\infty}|P_n(\mu_m(z))| \\
&+\left(\frac{1}{1+|a_0|}+\frac{r^m}{1-r^m}\right)\sum_{n=1}^{\infty}|P_n(\mu_m(z))|^2+\lambda\frac{\mathcal S_|\mu_m(z)|(f)}{\pi}\\
&=|F(\mu_m(z))|+\frac{r^m}{1-r^{2m}}\bigl(1-|F(z)|^2\bigr)
+\sum_{n=2}^{\infty}(1-a^2)(r^m)^n
\\
&\quad+\left(\frac{1}{1+|a_0|}+\frac{r^m}{1-r^m}\right)\sum_{n=1}^{\infty}(1-a^2)^2(r^m)^{2n} +\lambda \frac{S_r}{\pi} \\
&\le \frac{a+r^m}{1+ar^m}
+\frac{r^m}{1-r^{2m}}
\left[1-\left(\frac{a+r^m}{1+ar^m}\right)^2\right]
+\frac{(1-a^2)r^{2m}}{1-r^m}
+\lambda \frac{(1-a^2)^2r^{2m}}{(1-a^2r^{2m})^2},
\qquad 0<r<1 \\
&= \frac{a+r^m}{1+ar^m}
+\frac{(1-a^2)r^m}{(1+ar^m)^2}
+\frac{(1-a^2)r^{2m}}{1-r^m}
+\lambda \frac{(1-a^2)^2r^{2m}}{(1-a^2r^{2m})^2} \\
&=: I(r).
\end{aligned}
\]
It is easy to see that $I(r)$ is an increasing function on $0\leq r\leq R_m$, where $R_m$ is the unique postive root of $r^m=\frac{\sqrt{17}-3}{4}$ in $(0,1)$. Then we have
\[
\begin{aligned}
I(r)\leq I(R_m)&=
\frac{(\sqrt{17}-3)+4a}
     {4+(\sqrt{17}-3)a}
+\frac{(\sqrt{17}-3)^2(1-a^2)}
       {4(7-\sqrt{17})}
+\frac{4(\sqrt{17}-3)(1-a^2)}
       {\bigl(4+(\sqrt{17}-3)a\bigr)^2}
\\
&\quad
+\frac{8\lambda(13-3\sqrt{17})(1-a^2)^2}
{\left(8+(-13+3\sqrt{17})a^2\right)^2}.
\end{aligned}
\]
The remainder of the proof follows along the same lines as that of \cite[Theorem 2.3]{AR2025}. Therefore, we omit the routine details and obtain the desired conclusion.

\medskip

To establish the sharpness of the constant $\lambda=\frac{221-43\sqrt{17}}{64},$
we consider an extremal family of holomorphic mappings.

Fix \(z_0\in\partial B_X\). For \(a\in(0,1)\), define
\begin{equation}\label{F2}
F(z)=f(l_{z_0}(z)),\qquad z\in B_X,
\end{equation}
where $f(\zeta)=\frac{a+\zeta}{1+a\zeta},
\zeta\in\mathbb D,$
and \(l_{z_0}\in T(z_0)\). Let $\mu_m(z)=l_{z_0}(z)^{m-1}z, m\ge1.$
Then, for \(r\in(0,1)\),
\[
F(rz_0)=\frac{a+r}{1+ar}
      =a+\sum_{s=1}^{\infty}P_s(rz_0),
\]
where $P_s(z_0)=(1-a^2)(-a)^{s-1}, s\ge1.$ Consequently, $P_s(\mu_m(rz_0)) =(1-a^2)(-a)^{s-1}r^{ms}.$

Moreover,
\[
DF(z)(z)
=
\frac{(1-a^2)l_{z_0}(z)}
     {(1+al_{z_0}(z))^2}.
\]

Now let \(z=R_mz_0\), where $R_m^m=\frac{\sqrt{17}-3}{4}$ and let \(\varepsilon>0\). A straightforward computation shows that
\[
\begin{aligned}
\mathcal{G}_{F_2}(R_mz_0)
&=
|F(\mu_m(R_mz_0))|
+|DF(\mu_m(R_mz_0))(\mu_m(R_mz_0))|
+\sum_{n=2}^{\infty}|P_n(\mu_m(R_mz_0))|
\\
&\quad
+\left(\frac{1}{1+|a_0|}
+\frac{R_m^m}{1-R_m^m}\right)
\sum_{n=1}^{\infty}|P_n(\mu_m(R_mz_0))|^2
+\left(\frac{221-43\sqrt{17}}{64}
+\varepsilon\right)
\frac{\mathcal S_{|\mu_m(R_mz_0)|}(F)}{\pi}
\\
&=
\frac{R_m^m+a}{1+aR_m^m}
+\frac{(1-a^2)R_m^m}
       {(1+aR_m^m)^2}
+\frac{(1-a^2)aR_m^{2m}}
       {(1-aR_m^m)(1+aR_m^m)}
+\frac{1+aR_m^m}
       {(1+a)(1-R_m^m)}
\,\frac{(1-a^2)^2R_m^{2m}}
       {1-a^2R_m^{2m}}
\\
&\quad
+\left(\frac{221-43\sqrt{17}}{64}
+\varepsilon\right)
\frac{(1-a^2)^2R_m^{2m}}
     {(1-a^2R_m^{2m})^2}
\\
&=
1+
\frac{(1-a)^3F_2(a)}
     {(7-\sqrt{17})
      \bigl(-16+(26-6\sqrt{17})a^2\bigr)^2}
+\varepsilon
\frac{64(1-a^2)^2(71-17\sqrt{17})}
     {(7-\sqrt{17})
      \bigl(-16+(26-6\sqrt{17})a^2\bigr)^2},
\end{aligned}
\]
where
\[
\begin{aligned}
F_2(a)
={}&
(23126-5658\sqrt{17})
+2(28769-6983\sqrt{17})a
\\
&\quad
+24(2041-495\sqrt{17})a^2
+8(2041-495\sqrt{17})a^3.
\end{aligned}
\]

As \(a\to1^{-}\), the first correction term is of order \((1-a)^3\), whereas the second is of order \((1-a)^2\). Therefore,
\[
\mathcal{G}_{F_2}(R_mz_0)=1+C(1-a)^2\varepsilon
+o\bigl((1-a)^2\bigr),
\]
where $C=\frac{256(71-17\sqrt{17})}{(7-\sqrt{17})\bigl(-16+(26-6\sqrt{17})\bigr)^2}>0.$

Hence, $\mathcal{G}_{1,F}(R_mz_0)>1$  for all \(a\) sufficiently close to \(1\). Consequently, the asserted inequality fails whenever the coefficient of the area term exceeds $\frac{221-43\sqrt{17}}{64}.$

Therefore, the constant $\lambda=\frac{221-43\sqrt{17}}{64}$
is best possible. This completes the proof.

\end{proof}

\section*{Declarations}

\subsection*{Funding}
The second author acknowledges financial support from the Council of Scientific and Industrial Research (CSIR), New Delhi, India under Grant No. 09/1224(16975)/2023-EMR-I.

\subsection*{Data Availability Statement}
Not applicable.

\subsection*{Conflict of Interest}
The author declares that there is no conflict of interest.

\subsection*{Clinical Trial Number} not applicable.
\subsection*{Author Contribution} All authors contributed equally to this work

\end{document}